\documentclass[11pt]{article}

\usepackage[T1]{fontenc}
\usepackage[utf8]{inputenc}
\usepackage[english]{babel}
\usepackage{lmodern}
\usepackage[a4paper,margin=1in]{geometry}
\usepackage{amsmath,amssymb,amsthm,mathtools}
\usepackage{microtype}
\usepackage[hidelinks]{hyperref}
\usepackage{enumitem}

\allowdisplaybreaks
\newtheorem{theorem}{Theorem}[section]
\newtheorem{lemma}[theorem]{Lemma}
\newtheorem{corollary}[theorem]{Corollary}
\newtheorem{proposition}[theorem]{Proposition}
\newtheorem{conjecture}[theorem]{Conjecture}
\theoremstyle{definition}
\newtheorem{definition}[theorem]{Definition}
\newtheorem{construction}[theorem]{Construction}
\theoremstyle{remark}
\newtheorem{remark}[theorem]{Remark}

\newcommand{\mad}{\operatorname{mad}}
\newcommand{\dout}{d^{+}}
\newcommand{\pchi}{\vec{\chi}}
\newcommand{\pot}{\pi}
\newcommand{\ceil}[1]{\left\lceil #1\right\rceil}
\newcommand{\floor}[1]{\left\lfloor #1\right\rfloor}
\newcommand{\defect}{\operatorname{defect}}

\hypersetup{
  pdftitle={Bounded-Defect Proper Orientations and Graphs on Surfaces},
  pdfauthor={Qi Wu and Yong Lu},
  pdfsubject={Proper orientations, defective colorings, and graphs on surfaces},
  pdfkeywords={proper orientation, maximum average degree, defective coloring, graphs on surfaces}
}

\title{\Large Bounded-Defect Proper Orientations and Graphs on Surfaces}

\author{Qi Wu, Yong Lu\thanks{Corresponding author.}\\[2pt]
\small School of Mathematics and Statistics, Jiangsu Normal University,\\[-1pt]
\small Xuzhou, Jiangsu 221116, People's Republic of China\\[-1pt]
\small Emails:~\texttt{wuqimath@163.com, luyong@jsnu.edu.cn}}
\date{}

\begin{document}
\maketitle

\begin{abstract}
Chen, Mohar and Wu [J. Combin. Theory Ser. B 161 (2023)] conjectured that every graph of orientable genus $g$ has proper orientation number $O(\sqrt g)$. We prove this conjecture. Our main result extends the potential-outdegree method from independent color classes to three parts of bounded internal degree. Let $k,D\in\mathbb Z_{\ge0}$. If $G$ has maximum average degree at most $2k$ and $V(G)=V_1\cup V_2\cup V_3$, where each $V_i$ induces a subgraph of maximum degree at most $D$, then
$$
   \pchi(G)\le k+9D+6+
   \floor{\frac32\ceil{\frac{2D+2}{5}}}
   \le k+\floor{\frac{48D}{5}}+8.
$$
When the partition is given, we construct such an orientation in polynomial time. For every $D\ge1$, we also construct a graph with a partition into three parts of internal maximum degree at most $D$ for which $\pchi(G)-\ceil{\mad(G)/2}\ge6D+5$; hence the linear dependence on $D$ is unavoidable. Finally, Woodall's defective-coloring theorem yields $\pchi(G)=O(\sqrt{\gamma+1})$ for every graph of Euler genus $\gamma$.
\end{abstract}



\section{Introduction}

All graphs in this paper are finite and simple. For $S\subseteq V(G)$,  write $G[S]$ for the subgraph induced by $S$ and $N_G(S)=\{v\in V(G)\setminus S:v\text{ has a neighbor in }S\}$. Let $N_G(v)=N_G(\{v\})$ and omit the subscript when the ambient graph is clear. For integers $a\le b$,  write $[a,b]=\{a,a+1,\ldots,b\}$.

Let $Q$ be an orientation of a graph $G$. Write $\dout_Q(v)$ for the outdegree of $v$ and $\Delta^+(Q)=\max_{v\in V(G)}\dout_Q(v)$. The orientation $Q$ is \emph{proper} if $\dout_Q(u)\ne\dout_Q(v)$ for every edge $uv\in E(G)$. The \emph{proper orientation number} of $G$ is $\pchi(G)=\min\{\Delta^+(Q):Q\text{ is a proper orientation of }G\}$. Some authors use indegrees instead. Reversing every arc shows that the two conventions give the same parameter.

Borowiecki, Grytczuk and Pil\'sniak~\cite{BorowieckiGrytczukPilsniak2012} proved that every graph has a proper orientation. Since the outdegrees in a proper orientation give a proper vertex coloring, $\chi(G)-1\le\pchi(G)\le\Delta(G)$. Proper orientations belong to the wider family of degree-distinguishing problems. Karo\'nski, {\L}uczak and Thomason~\cite{KaronskiLuczakThomason2004} proposed the 1-2-3 Conjecture, which Keusch~\cite{Keusch2024} later proved. Ahadi and Dehghan~\cite{AhadiDehghan2013} initiated the systematic study of the proper orientation number. The decision problem remains hard even on several sparse graph classes; see, for example,~\cite{AraujoCohenRezendeHavetMoura2015}. Earlier work treated bipartite and chordal graphs and several subclasses of outerplanar and planar graphs; see~\cite{AiEtAl2020,AraujoCohenRezendeHavetMoura2015,AraujoEtAlChordal2021,AraujoHavetSalesSilva2016,KnoxEtAl2017,Noguchi2020}.

For a nonempty graph $G$, its \emph{maximum average degree} is $\mad(G)=\max\{2|E(H)|/|V(H)|:\varnothing\ne H\subseteq G\}$, where $H$ ranges over the nonempty subgraphs of $G$;  set $\mad(\varnothing)=0$. Every orientation $Q$ of $G$ satisfies $\Delta^+(Q)\ge\ceil{\mad(G)/2}$. Chen, Mohar and Wu~\cite[Theorem~1.2]{ChenMoharWu2023} proved that every bipartite graph satisfies $$\ceil{\mad(G)/2}\le\pchi(G)\le\ceil{\mad(G)/2}+3.$$ They also obtained the bounds $14$, $11$ and $10$ for planar graphs, $3$-colorable planar graphs and outerplanar graphs, respectively~\cite[Theorems~1.4 and~1.5]{ChenMoharWu2023}. More generally, every $r$-colorable graph satisfies the following bound~\cite[Theorem~1.8]{ChenMoharWu2023}:
\begin{equation}
   \pchi(G)\le \frac12\mad(G)+
   O\!\left(\frac{r\log r}{\log\log r}\right).
   \label{eq:cmw-color}
\end{equation}
Their proof introduced potential outdegrees and combined fractional orientations with a weighted matching argument~\cite[Sections~2 and~3]{ChenMoharWu2023}.

Wang, Wang and Yu~\cite[Theorem~4]{WangWangYu2026} subsequently proved that every $3$-partite graph satisfies $\pchi(G)\le\ceil{\mad(G)/2}+7$. They~\cite[Section~2]{WangWangYu2026} also constructed $r$-partite graphs for which the additive term is at least $\floor{5r/2}-2$. Their proof of the upper bound retains the potential-outdegree and matching framework but avoids fractional orientations~\cite[Section~3]{WangWangYu2026}.

A partition $\mathcal P=(V_1,\ldots,V_s)$ of $V(G)$ is a \emph{$(D_1,\ldots,D_s)$-coloring} if $\Delta(G[V_i])\le D_i$ for every $i$. Write $\defect(\mathcal P)=\max_{1\le i\le s}\Delta(G[V_i])$;  this number is the \emph{defect} of $\mathcal P$. Thus $\mathcal P$ is an $s$-coloring with defect at most $D$ exactly when $\defect(\mathcal P)\le D$. Cowen, Cowen and Woodall~\cite{CowenCowenWoodall1986} proved that every planar graph has a $(2,2,2)$-coloring. Woodall~\cite{Woodall2011} established defective-coloring and defective-choosability bounds for graphs on surfaces; see also~\cite{ChoiEsperet2019}.

Each graph $G[V_i]$ is $(D+1)$-colorable, so refining the three defect-$D$ parts produces at most $3(D+1)$ independent sets. Theorem~5.1 of Chen, Mohar and Wu~\cite[Theorem~5.1]{ChenMoharWu2023} therefore already gives an additive term of order $D\log D/\log\log D$. Our contribution is to replace that superlinear dependence by an explicit linear term. The three original parts are kept and their internal edges are accounted for directly. The key point is a variable matching capacity that depends on the current degree of a selected vertex inside the active subgraph of its defect part.

For $D\in\mathbb Z_{\ge0}$, set $F(D)=9D+6+\floor{\frac32\ceil{(2D+2)/5}}$. Our main technical result is as follows.

\begin{theorem}\label{thm:defective-main}
Let $k,D\in\mathbb Z_{\ge0}$. Suppose that $\mad(G)\le2k$ and that $V(G)=V_1\cup V_2\cup V_3$ is a partition with $\Delta(G[V_i])\le D$ for $i=1,2,3$. Then $\pchi(G)\le k+F(D)\le k+\floor{48D/5}+8$. If the partition is given, we can find such an orientation in polynomial time.
\end{theorem}

Taking $k=\ceil{\mad(G)/2}$ gives a bound depending only on $G$ and the partition. For $D=0$, Theorem~\ref{thm:defective-main} recovers the $+7$ bound of Wang, Wang and Yu~\cite[Theorem~4]{WangWangYu2026}. The first three values are $F(0)=7$, $F(1)=16$, and $F(2)=27$.

We adapt the weighted independent-set and capacitated Hall arguments of Chen, Mohar and Wu~\cite[Section~3]{ChenMoharWu2023} and Wang, Wang and Yu~\cite[Section~3]{WangWangYu2026}. Our capacity at a selected vertex varies with its current number of active neighbors in the same defect part. This dependence lets us retain the three-part structure instead of refining each part into many independent sets.

For the lower bound, we retain the forcers and the density-control lemma of Wang, Wang and Yu~\cite[Section~2]{WangWangYu2026}. The new part is a three-block core whose internal degree is exactly controlled, together with an equal-multiplicity count for the remaining outdegree values. The order of the error term cannot be improved.

\begin{theorem}\label{thm:lower}
For every integer $D\ge1$, there is a graph $G_D$ with a partition $V(G_D)=V_1\cup V_2\cup V_3$ satisfying $\Delta(G_D[V_i])\le D$ and $\pchi(G_D)-\ceil{\mad(G_D)/2}\ge6D+5$.
\end{theorem}

For $D\in\mathbb Z_{\ge0}$, let $f(D)$ be the least integer $c$ such that every graph admitting a $3$-coloring with defect at most $D$ satisfies $\pchi(G)\le\ceil{\mad(G)/2}+c$. Combining  Theorems~\ref{thm:defective-main} and~\ref{thm:lower}, we obtain  $6D+5\le f(D)\le F(D)$ for every $D\ge1$. For $D=0$, the current bounds are $5\le f(0)\le7$~\cite[Theorem~4 and Section~2]{WangWangYu2026}.

 Euler genus is used throughout. An orientable surface of genus $g$ has Euler genus $2g$, while a nonorientable surface of genus $g$ has Euler genus $g$. The Euler genus of a graph is the minimum Euler genus of a surface in which the graph embeds. Standard density estimates together with \eqref{eq:cmw-color} give an upper bound of order $\sqrt\gamma\log\gamma/\log\log\gamma$ for graphs of Euler genus $\gamma$. In the original formulation below, genus means orientable genus.

\begin{conjecture}[{Chen--Mohar--Wu~\cite[Conjecture~7.4]{ChenMoharWu2023}}]\label{conj:cmw}
If $G$ has orientable genus $g$, then $\pchi(G)=O(\sqrt g)$.
\end{conjecture}

Combining Theorem~\ref{thm:defective-main} with Woodall's defective-coloring theorem~\cite[Theorem~2(a)(i)]{Woodall2011}, we prove the conjectured order of magnitude.

\begin{theorem}\label{thm:surface}
There is an absolute constant $C$ such that every graph $G$ of Euler genus $\gamma$ satisfies $\pchi(G)\le C\sqrt{\gamma+1}$. More precisely, for $\gamma\ge1$, let $D_\gamma=\ceil{\max\{9,2+\sqrt{4\gamma+6}\}}$. Then $\pchi(G)\le\ceil{3+\tfrac12\sqrt{6\gamma}}+F(D_\gamma)$.
\end{theorem}

Since an orientable surface of genus $g$ has Euler genus $2g$, Theorem~\ref{thm:surface} implies Conjecture~\ref{conj:cmw}. The orientation tools are developed in Section~2 and the bounded-defect lemmas in Section~3. The upper bound is proved in Section~4. Section~5 contains the lower construction and the application to graphs on surfaces.

\section{Orientation tools}

\begin{definition}\label{def:partial-orientation}
A \emph{partial orientation} $p$ assigns a direction to some edges of $G$. Write $\dout_p(v)$ for the number of edges currently directed out of $v$ and $u_p(v)$ for the number of incident unoriented edges. The \emph{potential outdegree} of $v$ is $\pot_p(v)=\dout_p(v)+u_p(v)$. A vertex is \emph{settled} when all its incident edges have been oriented; its outdegree is then final. For an integer $j$, let $A_j(p)$ be the set of settled vertices of outdegree $j$. We call $j$ a \emph{level}, and  call it \emph{free} if $A_j(p)=\varnothing$. A \emph{prescription} assigns a direction to an incident unoriented edge of an unsettled vertex; we impose all prescriptions when that vertex is settled.
\end{definition}

When the ambient orientation is clear,  $\dout(v)$ is written instead of $\dout_Q(v)$. An oriented edge incident with $v$ is called an \emph{outedge} or an \emph{inedge} of $v$ according to its direction. If an ambient orientation is clear and $W\subseteq V(G)$, then $d_W^+(v)$ denotes the number of arcs from $v$ to vertices of $W$. If $H$ is a subgraph, $d_H^+(v)$ denotes the outdegree of $v$ in the restriction to $H$. Directing an unoriented edge $uv$ from $u$ to $v$ leaves $\pot_p(u)$ unchanged and decreases $\pot_p(v)$ by one.

\begin{lemma}[Hakimi~\cite{Hakimi1965}]\label{lem:hakimi}
Let $k\in\mathbb Z_{\ge0}$. A graph $G$ has an orientation $D_0$ with $\Delta^+(D_0)\le k$ if and only if $\mad(G)\le2k$.
\end{lemma}

Whenever $\mad(G)\le2k$,  a fixed orientation $D_0$ with $\Delta^+(D_0)\le k$ is called a \emph{base orientation}.

 The following standard capacitated form of Hall's theorem will be used. Vertices of capacity zero may be deleted, so the statement is also the special case of the Hall--Ore theorem recorded by Wang, Wang and Yu.

\begin{lemma}[{Capacitated Hall--Ore theorem; see~\cite{Hall1935,Ore1956,Ore1957} and~\cite[Lemma~7]{WangWangYu2026}}]\label{lem:capacitated-hall}
Let $B$ be a bipartite graph with parts $X$ and $Y$, and let $b:Y\to\mathbb Z_{\ge0}$. There is a set $\mathcal M\subseteq E(B)$ such that every vertex of $X$ is incident with exactly one edge of $\mathcal M$ and every $y\in Y$ is incident with at most $b(y)$ edges of $\mathcal M$ if and only if $|S|\le\sum_{y\in N_B(S)}b(y)$ for every $S\subseteq X$.
\end{lemma}

The set $\mathcal M$ can be found by the standard maximum-flow reduction in which the vertices of $X$ have demand one and each $y\in Y$ has capacity $b(y)$.

With the notation of Definition~\ref{def:partial-orientation}, assume from now on that $\mad(G)\le2k$, and fix a base orientation $D_0$ given by Lemma~\ref{lem:hakimi}. The following conditions are maintained throughout the construction.
\begin{enumerate}[label=(P\arabic*)]
\item\label{inv:settled-edge} Every oriented edge has at least one settled endpoint.
\item\label{inv:base} If $v$ is unsettled, then every edge currently directed out of $v$ has the same direction in $D_0$. In particular, $\dout_p(v)\le k$.
\item\label{inv:levels} For every integer $j$, the set $A_j(p)$ is independent.
\end{enumerate}

The next lemma isolates the settlement operation used in the level constructions of Chen, Mohar and Wu~\cite[Section~3]{ChenMoharWu2023} and Wang, Wang and Yu~\cite[Section~3]{WangWangYu2026}.

\begin{lemma}[{Settlement step; cf.~\cite[Section~3]{ChenMoharWu2023} and~\cite[Section~3]{WangWangYu2026}}]\label{lem:fix}
Let $A$ be an independent set of unsettled vertices such that $\pot_p(a)\ge j$ for every $a\in A$, and suppose that $A\cup A_j(p)$ is independent. For each $a\in A$, let $F_a$ be a set of at most $B$ unoriented edges incident with $a$. If $j\ge k+B$, then we may direct the edges in $F_a$ out of $a$ and settle every vertex of $A$ with outdegree $j$, while preserving~\ref{inv:settled-edge}--\ref{inv:levels}.
\end{lemma}

\begin{proof}
We use the base-orientation bookkeeping from the cited level constructions~\cite[Section~3]{ChenMoharWu2023} and~\cite[Section~3]{WangWangYu2026}. Fix $a\in A$. By~\ref{inv:base}, the edges already directed out of $a$, together with the unoriented edges directed out of $a$ in $D_0$, form a set of at most $k$ edges. Direct the latter edges out of $a$, and also direct every edge of $F_a$ out of $a$. After removing repetitions, we have prescribed at most $k+B\le j$ outedges.

Let $q$ be the number of edges already directed or prescribed out of $a$, and let $u$ be the number of incident edges that remain unoriented. Then $q\le j$ and $q+u=\pot_p(a)\ge j$. Direct exactly $j-q$ of these edges out of $a$ and all the others into $a$. The final outdegree of $a$ is $j$.

The prescriptions do not conflict because $A$ is independent. Every newly oriented edge has a settled endpoint. If an unsettled neighbor $x$ receives a new outedge $xa$, then $ax$ was not directed out of $a$ in $D_0$, so $xa$ agrees with $D_0$. Finally, $A\cup A_j(p)$ is independent. Hence~\ref{inv:settled-edge}--\ref{inv:levels} are preserved.
\end{proof}

We use the following completion argument at the end of the construction.

\begin{lemma}[{Greedy completion; cf.~\cite[Lemma~3.2]{ChenMoharWu2023} and~\cite[Section~3]{WangWangYu2026}}]\label{lem:greedy-completion}
Suppose that a partial orientation satisfies~\ref{inv:settled-edge} and~\ref{inv:levels}. Assume that every settled vertex has outdegree at least $L+1$ and every unsettled vertex has potential at most $L$. Then the partial orientation extends to a proper orientation with maximum outdegree at most $\max\{L,\max\{\dout_p(v):v\text{ is settled}\}\}$, where the second maximum is omitted if no vertex is settled.
\end{lemma}

\begin{proof}
We use the integer greedy completion appearing in the cited arguments~\cite[Lemma~3.2]{ChenMoharWu2023} and~\cite[Section~3]{WangWangYu2026}. While an unsettled vertex remains, choose a vertex $v$ of maximum potential, direct every still unoriented edge incident with $v$ out of $v$, and settle $v$. Its final outdegree is the potential it had when chosen and is at most $L$.

Consider an edge $uv$ whose endpoints are settled by this procedure, with $u$ chosen before $v$. Immediately before $u$ is settled, the edge $uv$ is unoriented by~\ref{inv:settled-edge}; we then direct it from $u$ to $v$. If $\alpha$ is the potential of $u$ at that moment, the potential of $v$ is at most $\alpha$ and drops by one. It never increases afterwards, so $\dout(v)\le\alpha-1<\alpha=\dout(u)$. Thus adjacent vertices settled by the procedure receive different outdegrees. A previously settled vertex has outdegree at least $L+1$, while every newly settled vertex has outdegree at most $L$. Two previously settled adjacent vertices have different outdegrees by~\ref{inv:levels}. Hence the completed orientation is proper.
\end{proof}

\section{Bounded-defect lemmas}

We begin with a bounded-defect version of the gap-capping procedure of Chen, Mohar and Wu~\cite[Lemma~3.4]{ChenMoharWu2023}.

\begin{lemma}[{Bounded-defect gap capping; cf.~\cite[Lemma~3.4]{ChenMoharWu2023}}]\label{lem:block-capping}
Let $T\subseteq V(G)$ satisfy $\Delta(G[T])\le D$. Suppose that the levels $J,J-1,\ldots,J-D$ are free and that $J-D\ge k+D$. We can use these levels so that every unsettled vertex $x\in T$ has potential at most $J-D-1$ afterwards.
\end{lemma}

\begin{proof}
We follow the level-by-level capping scheme of Chen, Mohar and Wu~\cite[Lemma~3.4]{ChenMoharWu2023}, replacing an independent color class by maximal independent sets in the successive active subgraphs. For $j=J,J-1,\ldots,J-D$, let $X_j=\{x\in T:x\text{ is unsettled and }\pot_p(x)\ge j\}$. Choose a maximal independent set $I_j$ in $G[X_j]$. For every $a\in I_j$, prescribe each edge $ax$ with $x\in X_j\setminus I_j$ to point from $a$ to $x$. There are at most $D$ such edges at $a$. Since $j\ge k+D$, Lemma~\ref{lem:fix} settles $I_j$ at level $j$.

Let $x\in T$ remain unsettled after all $D+1$ rounds. Whenever $x\in X_j$, maximality of $I_j$ gives a neighbor of $x$ in $I_j$, and $x$ receives an inedge in that round. These neighbors are distinct in different rounds. Since $x$ has at most $D$ neighbors in $T$, there is a round in which $x\notin X_j$.

Let $J-t$ be the first such level. After that round, $\pot_p(x)\le J-t-1$. We prove by induction that, for every $s$ with $t\le s\le D$, the potential after the round at level $J-s$ is at most $J-s-1$. The assertion holds for $s=t$. Suppose that $s<D$ and consider the next round, at level $J-s-1$. Before this round, $\pot_p(x)\le J-s-1$. If $x\notin X_{J-s-1}$, integrality gives $\pot_p(x)\le J-s-2$. If $x\in X_{J-s-1}$, then $x$ receives an inedge from $I_{J-s-1}$, and its potential again becomes at most $J-s-2$. This proves the induction step. Taking $s=D$ gives the result.
\end{proof}

We adapt the weighted independent-set mechanism developed by Chen, Mohar and Wu~\cite[Section~3]{ChenMoharWu2023} and refined by Wang, Wang and Yu~\cite[Section~3]{WangWangYu2026}. In contrast with the independent-part setting, the capacity at a selected vertex depends on its degree in the current active subgraph of the defect part.

\begin{lemma}[{Variable-load protected level; cf.~\cite[Lemma~3.5]{ChenMoharWu2023} and~\cite[Section~3]{WangWangYu2026}}]\label{lem:variable-level}
Let $T,P\subseteq V(G)$ be disjoint, let $\Delta(G[T])\le D$, and fix a proper coloring $P=P_1\cup\cdots\cup P_r$ with $r\ge1$. Let $j$ be a free level. Assume that every unsettled vertex of $P$ has potential at most $j$ and that $j\ge k+D$.

Set
\[
   X=\{x\in T:x\text{ is unsettled and }\pot_p(x)\ge j\},\qquad
   Y=\{y\in P:y\text{ is unsettled and }\pot_p(y)=j\}.
\]
The set $X$ is the active set of $T$ at level $j$. For $x\in X$, set $d_X(x)=|N(x)\cap X|$ and $\lambda(x)=\floor{(j-k-d_X(x))/r}$. Then we can use level $j$ so that:
\begin{enumerate}[label=(\roman*)]
\item\label{item:P-drop} every unsettled vertex of $P$ has potential at most $j-1$;
\item\label{item:T-drop} if $x\in X$ remains unsettled, then either $x$ receives an inedge from a vertex of $T$ settled at level $j$, or its potential decreases by at least $\lambda(x)+1$.
\end{enumerate}
Moreover,~\ref{inv:settled-edge}--\ref{inv:levels} remain valid, and the operation can be carried out in polynomial time from the given coloring of $G[P]$.
\end{lemma}

\noindent\emph{Notation for the proof.} Let $n=|V(G)|$. If $A\subseteq X\cup Y$ is independent, write $A_X=A\cap X$ and $A_Y=A\cap Y$, and set $w(A)=\sum_{x\in A_X}\lambda(x)+|A_Y|$ and $\Phi(A)=(n+1)w(A)+|A_X|$. We also write $Y_A=\{y\in Y\setminus A_Y:N(y)\cap A_Y=\varnothing\}$.

\begin{proof}
We modify the weighted independent-set exchanges and capacitated matching arguments in Chen, Mohar and Wu~\cite[Lemma~3.5]{ChenMoharWu2023} and Wang, Wang and Yu~\cite[Section~3]{WangWangYu2026}; the new feature is the active-degree term $d_X(x)$. Since $d_X(x)\le D\le j-k$, every $\lambda(x)$ is nonnegative. Start with $A=\varnothing$ and perform either of the following exchanges whenever possible.

First, suppose that $x\in X\setminus A_X$ has no neighbor in $A_X$ and $|N(x)\cap A_Y|\le\lambda(x)$. Delete $N(x)\cap A_Y$ from $A$ and add $x$. The new set is independent. Its weight increases; or its weight is unchanged and $|A_X|$ increases. In either case, $\Phi$ increases.

Second, suppose that, for some $i$, a set $S\subseteq Y_A\cap P_i$ satisfies $|S|>\sum_{a\in N(S)\cap A_X}\lambda(a)$. Delete $N(S)\cap A_X$ from $A$ and add $S$. Since $S$ is independent and has no neighbor in $A_Y$, the new set is independent. Its weight increases by at least one, while $|A_X|$ decreases by at most $n$. Hence $\Phi$ increases.

We find the first exchange by scanning $X$. For a fixed $i$, the second exchange exists precisely when the capacitated Hall condition fails in the bipartite graph between $Y_A\cap P_i$ and $A_X$, with capacity $\lambda(a)$ at $a\in A_X$. A maximum-flow computation either certifies the condition or returns a set $S$ that violates it. If $X\cup Y\ne\varnothing$, then $j\le n-1$, since a potential never exceeds the degree of a vertex. Thus $\lambda(a)\le n-1$, $w(A)=O(n^2)$ and $\Phi(A)=O(n^3)$. Each exchange increases the integer $\Phi$ by at least one, so the process terminates after polynomially many exchanges. If $X\cup Y=\varnothing$, there is nothing to do.

At termination,
\begin{align}
&x\in X\setminus A_X,\quad N(x)\cap A_X=\varnothing
   \quad\Longrightarrow\quad |N(x)\cap A_Y|\ge\lambda(x)+1,
   \label{eq:variable-X}\\
&|S|\le\sum_{a\in N(S)\cap A_X}\lambda(a)
   \quad(S\subseteq Y_A\cap P_i,\ 1\le i\le r).
   \label{eq:variable-color}
\end{align}

First settle $A_Y$ at level $j$ by directing every unoriented edge incident with $A_Y$ away from $A_Y$. For $y\in A_Y$, the resulting outdegree is $\dout_p(y)+u_p(y)=\pot_p(y)=j$. The set $A_Y$ is independent and level $j$ is free, so~\ref{inv:levels} is preserved. Every newly oriented edge has a settled endpoint, which preserves~\ref{inv:settled-edge}. An unsettled neighbor of $A_Y$ receives only an inedge, so no new outedge is created at an unsettled vertex and~\ref{inv:base} is unchanged.

For $S\subseteq Y_A$, let $S_i=S\cap P_i$. By \eqref{eq:variable-color},
\[
   |S|=\sum_{i=1}^r|S_i|
   \le\sum_{i=1}^r\sum_{a\in N(S_i)\cap A_X}\lambda(a)
   \le\sum_{a\in N(S)\cap A_X}r\lambda(a).
\]
Lemma~\ref{lem:capacitated-hall} gives a set $\mathcal M$ of edges between $Y_A$ and $A_X$ such that every vertex of $Y_A$ is incident with one edge of $\mathcal M$ and every $a\in A_X$ is incident with at most $r\lambda(a)$ edges of $\mathcal M$.

By~\ref{inv:settled-edge}, every edge of $\mathcal M$ and every edge between $A_X$ and $X\setminus A_X$ is still unoriented, since both endpoints are unsettled. For $a\in A_X$, prescribe every incident edge of $\mathcal M$ to point from $a$ to $Y_A$, and prescribe every edge $ax$ with $x\in X\setminus A_X$ to point from $a$ to $x$. The number of prescribed outedges at $a$ is at most $r\lambda(a)+d_X(a)\le j-k$.

Since $A$ is independent, settling $A_Y$ does not change the potential of a vertex in $A_X$. Thus $\pot_p(a)\ge j$ for every $a\in A_X$. Because level $j$ was free before we settled $A_Y$, we now have $A_j(p)=A_Y$; hence $A_X\cup A_j(p)=A_X\cup A_Y$ is independent. Lemma~\ref{lem:fix}, with $B=j-k$, therefore settles $A_X$ at level $j$.

Every vertex of $Y$ is settled, has a neighbor in $A_Y$, or belongs to $Y_A$ and receives an inedge along $\mathcal M$. Vertices of $P$ whose potential was below $j$ remain below $j$. This proves~\ref{item:P-drop}.

Let $x\in X\setminus A_X$ remain unsettled. If $x$ has a neighbor in $A_X$, it receives an inedge from a vertex of $T$ settled at level $j$. Otherwise, \eqref{eq:variable-X} gives at least $\lambda(x)+1$ neighbors of $x$ in $A_Y$, so its potential decreases by at least $\lambda(x)+1$. This proves~\ref{item:T-drop}. Lemma~\ref{lem:fix} preserves~\ref{inv:settled-edge}--\ref{inv:levels} when $A_X$ is settled.
\end{proof}

\begin{definition}
Let $D\in\mathbb Z_{\ge0}$, and suppose that Lemma~\ref{lem:variable-level} is applied repeatedly with a fixed set $T$ satisfying $\Delta(G[T])\le D$. Each application at one level is a \emph{protected round}. For an unsettled vertex $x\in T$,  a protected round is \emph{first-type for $x$} if it realizes the inedge alternative in Lemma~\ref{lem:variable-level}\ref{item:T-drop}; otherwise it is \emph{second-type for $x$}. If the round is conducted at level $j$, the \emph{gap of $x$} is $\max\{0,\pot_p(x)-j\}$. We use the convention that an empty sum is zero.
\end{definition}

A consecutive block of levels is obtained by applying Lemma~\ref{lem:variable-level} repeatedly.

\begin{corollary}\label{cor:variable-block}
Let $T,P\subseteq V(G)$ be disjoint, let $\Delta(G[T])\le D$, and fix a proper $r$-coloring of $G[P]$, where $r\ge1$. Suppose that every unsettled vertex of $P$ has potential at most $j_0$, and every unsettled vertex of $T$ has potential at most $j_0+\Gamma$, where $\Gamma\in\mathbb Z_{\ge0}$. Let $m\in\mathbb Z_{\ge0}$ and set $\sigma_t=\floor{(j_0-k-D-t)/r}$ for $0\le t<m$. Assume that $\sum_{t=0}^{m-1}\sigma_t\ge\Gamma$, that every level in $[j_0-D-m+1,j_0]$ is free, and, when $D+m>0$, that $j_0-D-m+1\ge k+D$. Then we can use these levels so that every unsettled vertex of $T\cup P$ has potential at most $j_0-D-m$. The construction is polynomial.
\end{corollary}

\begin{proof}
If $D+m=0$, then the assumed sum condition gives $D=m=\Gamma=0$, and there is nothing to prove. Assume $D+m>0$. For $0\le t<m$, the lowest-level condition gives $j_0-k-D-t\ge j_0-k-D-m+1\ge D\ge0$, so every $\sigma_t$ is nonnegative.

Apply Lemma~\ref{lem:variable-level} successively at the levels $j_0,j_0-1,\ldots,j_0-D-m+1$. Property~\ref{item:P-drop} keeps every unsettled vertex of $P$ below the next level.

Fix $x\in T$ that remains unsettled through all rounds. If, before some round at level $j$, we have $\pot_p(x)\le j$, then after that round $\pot_p(x)\le j-1$: this follows from integrality when $\pot_p(x)<j$, and from either alternative in Lemma~\ref{lem:variable-level}\ref{item:T-drop} when $\pot_p(x)=j$. The same argument then applies at every later level. It may therefore be assumed that $\pot_p(x)>j$ immediately before each round and track the positive gap between the potential and the current level.

At most $D$ rounds are first-type for $x$, because each gives $x$ an inedge from a distinct settled neighbor in $T$. Consider the $(t+1)$st second-type round for $x$, where $0\le t<m$, and let $b$ be the number of earlier rounds that are first-type for $x$. Its index among all rounds is $q=b+t$. The $b$ former active neighbors supplied by those first-type rounds have already been settled, so the current active degree satisfies $d_X(x)\le D-b$. At level $j_0-q$,
\[
   \lambda(x)\ge
   \floor{\frac{j_0-(b+t)-k-(D-b)}{r}}
   =\floor{\frac{j_0-k-D-t}{r}}=\sigma_t.
\]
A round that is first-type for $x$ decreases both the potential and the next level by at least one, so it does not increase the gap. A round that is second-type for $x$ decreases the potential by at least $\lambda(x)+1$ while the next level is lower by one; hence it decreases the gap by at least $\sigma_t$.

The initial gap is at most $\Gamma$. If the final gap were positive, then at least $m$ of the $D+m$ rounds would be second-type for $x$. After the first $m$ such rounds, the gap is at most $\Gamma-\sum_{t=0}^{m-1}\sigma_t\le0$, and no later round can increase it, a contradiction. Thus $\pot_p(x)\le j_0-D-m$ after the last round. The polynomial-time assertion follows from Lemma~\ref{lem:variable-level}.
\end{proof}

\section{The upper bound}

\begin{definition}
Let $D\in\mathbb Z_{\ge0}$. A \emph{capping block} is one application of Lemma~\ref{lem:block-capping} to a set $T$ satisfying $\Delta(G[T])\le D$. A \emph{protected block} is one application of Corollary~\ref{cor:variable-block}.
\end{definition}

\begin{proof}[Proof of Theorem~\ref{thm:defective-main}]
We keep the potential-outdegree framework of Chen, Mohar and Wu~\cite[Section~3]{ChenMoharWu2023} and Wang, Wang and Yu~\cite[Section~3]{WangWangYu2026}; Lemmas~\ref{lem:block-capping} and~\ref{lem:variable-level} supply the new treatment of internal edges in the defect parts. The statement is immediate for the empty graph, so put $n=|V(G)|\ge1$. If $k>n-1$, we replace $k$ by $n-1$; the inequality $\mad(G)\le2(n-1)$ remains valid, and the resulting conclusion is stronger. Let $\mathcal P=(V_1,V_2,V_3)$ and $D^*=\defect(\mathcal P)$. Since $D^*\le D$ and $F$ is nondecreasing, it is enough to prove the theorem with $D^*$ in place of $D$. We therefore replace $D$ by $D^*$ and assume $k,D\le n-1$.

Set $h=D+1$, $c=\ceil{2h/5}$, $q=\floor{c/2}$, and $M=k+9h+c+q-3$.
Fix a base orientation $D_0$ and a greedy proper $h$-coloring of each $G[V_i]$. We start with no oriented edge and no settled vertex. We use three capping blocks followed by two protected blocks, always taking consecutive levels from $M$ downwards.

Apply Lemma~\ref{lem:block-capping} first to $V_1$ at the levels $M,M-1,\ldots,M-h+1$, then to $V_2$ at the next $h$ levels, and finally to $V_3$ at the following $h$ levels. Put $j_0=M-3h=k+6h+c+q-3$. The lowest level in these blocks is $j_0+1=k+6h+c+q-2\ge k+D$, and every level is free when first used. Since potential never increases, after the three blocks every unsettled vertex satisfies
\begin{equation}
   \pot_p(v)\le
   \begin{cases}
      j_0+2h,&v\in V_1,\\
      j_0+h,&v\in V_2,\\
      j_0,&v\in V_3.
   \end{cases}
   \label{eq:three-caps-new}
\end{equation}

For the first protected block, apply Corollary~\ref{cor:variable-block} with $T=V_1$, $P=V_3$, $r=h$, $\Gamma=2h$, and $m_1=c$. For $0\le t<c$, the guaranteed decrease is $\sigma_t=\floor{(5h+c+q-2-t)/h}$. If $c\ge2$, then $q\ge1$ and $\sigma_t\ge5$ for every $t<c$, so $\sum_{t=0}^{c-1}\sigma_t\ge5c\ge2h$. If $c=1$, then $h\in\{1,2\}$ and the only term is $\floor{(5h-1)/h}=4\ge2h$.

This block has length $L_1=D+m_1=h-1+c$. Put $j_1=j_0-L_1=k+5h+q-2$. Its levels are $j_0,j_0-1,\ldots,j_1+1$; they are free because the capping blocks stopped at $j_0+1$, and $j_1+1=k+5h+q-1\ge k+D$. Corollary~\ref{cor:variable-block} therefore leaves every unsettled vertex of $V_1\cup V_3$ with potential at most $j_1$. By \eqref{eq:three-caps-new}, every unsettled vertex of $V_2$ has potential at most $j_0+h=j_1+2h-1+c$.

For the second protected block, take $T=V_2$, $P=V_1\cup V_3$, $r=2h$, and $\Gamma_2=2h-1+c$. We choose the fixed $h$-colorings of $V_1$ and $V_3$ with disjoint color sets. Set $a=\ceil{c/2}$ and $m_2=2h-1+a$. For $0\le t<m_2$, we have $\sigma_t=\floor{(j_1-k-D-t)/(2h)}=\floor{(4h+q-1-t)/(2h)}$.
For $0\le t<q$, we have $\sigma_t\ge2$. For $q\le t<m_2$, the numerator is smallest at $t=m_2-1$, where it equals $2h+q+1-a\ge2h$; hence $\sigma_t\ge1$. Therefore $\sum_{t=0}^{m_2-1}\sigma_t\ge2q+(m_2-q)=2h-1+a+q=2h-1+c=\Gamma_2$.
The block has length $L_2=D+m_2=3h-2+a$. Put $j_2=j_1-L_2$. Since $a=\ceil{c/2}$ and $q=\floor{c/2}$, we have $j_2-k=2h+q-a\in\{2h-1,2h\}$. The levels $j_1,j_1-1,\ldots,j_2+1$ are free, and $j_2+1\ge k+2h\ge k+D$. Corollary~\ref{cor:variable-block} leaves every unsettled vertex with potential at most $j_2$.

Every settled vertex has outdegree in $\{j_2+1,\ldots,M\}$. Lemma~\ref{lem:greedy-completion} completes the orientation with maximum outdegree at most $M$. Since $M-k=9D+6+c+\floor{c/2}=9D+6+\floor{3c/2}=F(D)$, this proves the first inequality.

For the simpler estimate, write $D=5z+r_0$, where $0\le r_0\le4$. Direct calculation gives
\[
F(D)=
\begin{cases}
48z+7,&r_0=0,\\
48z+16,&r_0=1,\\
48z+27,&r_0=2,\\
48z+36,&r_0=3,\\
48z+45,&r_0=4.
\end{cases}
\]
Thus $F(D)\le\floor{48D/5}+8$.

For the algorithmic assertion, compute $D^*=\defect(\mathcal P)$ from the supplied partition. For an integer $\kappa$, build a network with a source $s_0$, one node for each edge $e=uv$, one node for each vertex, and a sink $t_0$. Add an arc of capacity $1$ from $s_0$ to every edge-node, arcs of capacity $1$ from the node $e=uv$ to the nodes $u$ and $v$, and an arc of capacity $\kappa$ from each vertex-node to $t_0$. An integral flow of value $|E(G)|$ chooses one endpoint of every edge as its tail and gives an orientation of maximum outdegree at most $\kappa$; conversely, every such orientation gives a flow of value $|E(G)|$. Testing $\kappa=0,1,\ldots,n-1$ yields the least feasible value $\kappa_0$ and a corresponding base orientation. Lemma~\ref{lem:hakimi} gives $\kappa_0=\ceil{\mad(G)/2}\le k$. Running the construction with $\kappa_0$ and $D^*$ gives an orientation of maximum outdegree at most $\kappa_0+F(D^*)\le k+F(D)$.

The proper colorings and the maximal independent sets in Lemma~\ref{lem:block-capping} are found greedily. We obtain sets that violate the capacitated Hall condition, when they exist, and capacitated matchings in Lemma~\ref{lem:variable-level} by maximum flow. In the two protected blocks, the number of color classes is at most $2(D^*+1)\le2n$, so each exchange requires at most $O(n)$ flow tests. There are $O(n)$ levels. At a nontrivial protected level, $j\le n-1$, so the exchange potential $\Phi$ is $O(n^3)$ and permits at most $O(n^3)$ exchanges. Every flow network has polynomial size, and the final completion is greedy. Thus the construction runs in polynomial time.
\end{proof}

\section{Lower bounds and graphs on surfaces}

\subsection{A linear lower bound}

The terminology used in the lower construction is introduced first.

\begin{definition}\label{def:forcer}
Let $F$ be a graph with a distinguished set $A\subseteq V(F)$, let $I$ be a set of positive integers, and let $M\in\mathbb Z_{\ge0}$. We call $F$ an \emph{$I$-forcer for $A$ under the upper bound $M$} if, for every graph $H$ with $V(H)\cap V(F)=A$, every proper orientation $Q$ of $H\cup F$ with $\Delta^+(Q)\le M$ satisfies $\dout_Q(a)\notin I$ for every $a\in A$. We call the vertices of $A$ \emph{anchors} and the vertices of $V(F)\setminus A$ \emph{internal vertices}. An \emph{auxiliary orientation} of $F$ is an orientation of $F$ in which every anchor has outdegree zero.
\end{definition}

The ambient graph in Definition~\ref{def:forcer} may contain arbitrary edges with one or both ends in the anchor set. Whenever we use several forcers, we take their internal vertex sets pairwise disjoint. We use the bipartite gadget of Chen, Mohar and Wu~\cite[Section~4]{ChenMoharWu2023}, in the form recorded by Wang, Wang and Yu~\cite[Lemmas~1 and~4(i)]{WangWangYu2026}, and the interval gadget of Wang, Wang and Yu~\cite[Lemmas~2 and~4(ii)]{WangWangYu2026}. We retain the constructions because their explicit vertex sets are needed for the defective coloring below, but we quote their forcing and auxiliary-orientation properties rather than reproving them.

\begin{construction}\label{constr:basic-forcer}
Let $A$ be a set of $p\ge1$ vertices and let $M\ge p+1$. For each nonempty set $S\subseteq A$, add a set $B_S$ of $pM+1$ vertices, each adjacent precisely to the vertices of $S$. Take $pM+1$ pairwise disjoint $p$-sets $C_1,\ldots,C_{pM+1}$, join every vertex of every $C_t$ to every vertex of $A$, and add a vertex $d_t$ adjacent precisely to the vertices of $C_t$. Set
\[
   B=\bigcup_{\varnothing\ne S\subseteq A}B_S,\qquad
   C=\bigcup_{t=1}^{pM+1}C_t,\qquad
   \mathcal D=\{d_1,\ldots,d_{pM+1}\},
\]
and denote the resulting graph by $F_{\mathrm b}(A,p,M)$.
\end{construction}

\begin{lemma}[{Chen--Mohar--Wu~\cite[Section~4]{ChenMoharWu2023}; Wang--Wang--Yu~\cite[Lemmas~1 and~4(i)]{WangWangYu2026}}]\label{lem:basic-forcer}
Let $F_{\mathrm b}=F_{\mathrm b}(A,p,M)$ be the graph in Construction~\ref{constr:basic-forcer}. Then $F_{\mathrm b}$ is a bipartite $[1,p+1]$-forcer for $A$ under the upper bound $M$. It has an auxiliary orientation in which every internal vertex has outdegree at most $p$.
\end{lemma}

By Construction~\ref{constr:basic-forcer}, the subgraph induced by $A\cup C$ contains $K_{p,p(pM+1)}$ with one part equal to $A$.

\begin{definition}\label{def:forcing-block}
Let $R\ge3$ and $1\le\ell\le R-1$. An \emph{$(R,\ell)$-forcing block} is obtained as follows. For each $i\in[1,\ell]$, take a copy $Q_i$ of $K_{R+1}$ on $\{x_i,y_{i,0},\ldots,y_{i,R-1}\}$ and delete the edge $x_i y_{i,0}$. Add all edges among $x_1,\ldots,x_\ell$. The set $X=\{x_1,\ldots,x_\ell\}$ is called the \emph{interface clique} of the block.
\end{definition}

\begin{construction}\label{constr:interval-forcer}
Let $R\ge3$ and $1\le a\le b$, put $\ell=b-a+1\le R-1$, and let $p,M\in\mathbb Z_{\ge0}$. Set $s=p+a-R+1$, and assume that $s\ge2$ and $M\ge p+b$. For an anchor set $A$ of size $s$, take pairwise disjoint sets $Y_0,\ldots,Y_{R-1}$ of size $s-2$ and $N=M(s+R(s-2))+1$ pairwise disjoint $(R,\ell)$-forcing blocks from Definition~\ref{def:forcing-block}. In every block, join each $x_i$ to all vertices of $A$ and each $y_{i,j}$ to all vertices of $Y_j$. These new edges are the interface edges of the block. Denote the resulting graph by $F_{\mathrm i}(A;R,a,b,p,M)$.
\end{construction}

\begin{lemma}[{Wang--Wang--Yu~\cite[Lemmas~2 and~4(ii)]{WangWangYu2026}}]\label{lem:interval-forcer}
Let $F_{\mathrm i}=F_{\mathrm i}(A;R,a,b,p,M)$ be the graph in Construction~\ref{constr:interval-forcer}. Then $F_{\mathrm i}$ is a $[p+a,p+b]$-forcer for $A$ under the upper bound $M$. If
\begin{equation}
   s+\ceil{\frac{b-a}{2}}\le p,
   \qquad s-1+\ceil{\frac{R-1}{2}}\le p,
   \label{eq:forcer-orientation}
\end{equation}
then it has an auxiliary orientation in which every internal vertex has outdegree at most $p$.
\end{lemma}

\begin{remark}[Robustness under additional anchor edges]
The forcing proofs in the cited lemmas use only the global outdegree bound, the edges between anchors and internal vertices, and properness at internal neighbors. Extra edges with both ends in the anchor set therefore do not affect the forcing conclusions. The auxiliary orientations concern only the edges of the forcer. Thus the cited statements apply in the ambient setting of Definition~\ref{def:forcer}.
\end{remark}

We use the following density argument to combine several forcers.

\begin{lemma}[{Wang--Wang--Yu~\cite[Lemma~3]{WangWangYu2026}}]\label{lem:anchored-density}
Let $H$ be a graph, let $U\subseteq V(H)$, and let $p\in\mathbb Z_{\ge0}$. Suppose that $\mad(H[U])\le2p$ and that $H-E(H[U])$ has an orientation in which every vertex of $U$ has outdegree zero and every vertex outside $U$ has outdegree at most $p$. Then $\mad(H)\le2p$.
\end{lemma}

We use a common defective coloring for the two interval forcers in the construction below.

\begin{lemma}[{Defect coloring of the forcers in~\cite[Lemmas~1 and~2]{WangWangYu2026}}]\label{lem:forcer-defect-coloring}
Let $q\ge2$ and put $R=2q+1$. Suppose that all anchors receive color $0$. The basic forcer and each of the interval forcers with $(a,b)=(2,2q+1)$ or $(a,b)=(q,3q-1)$ admit a $3$-coloring with defect at most $q-1$. In the interval forcers, no internal neighbor of an anchor receives color $0$.
\end{lemma}

\begin{proof}
We color the two cited gadgets~\cite[Lemmas~1 and~2]{WangWangYu2026} in the parameter ranges used here. For the basic forcer, we give color $0$ to the part containing the anchors and color $1$ to the other part.

Consider either interval forcer. In both cases $\ell=b-a+1=2q$. In every forcing block, we color the clique $\{x_1,\ldots,x_{2q}\}$ with colors $1$ and $2$, using each color $q$ times. We give $Y_0$ color $0$, split $Y_1,\ldots,Y_{2q}$ into two groups of $q$ sets, and give the two groups colors $1$ and $2$, respectively.

Fix $i$. If $x_i$ has color $1$, we give $y_{i,0}$ color $1$, give $y_{i,j}$ color $2$ when $Y_j$ has color $1$, and give $y_{i,j}$ color $0$ when $Y_j$ has color $2$. If $x_i$ has color $2$, we interchange colors $1$ and $2$ in this rule.

The clique on $y_{i,0},\ldots,y_{i,2q}$ has color-class sizes $1,q,q$. The only $y$-vertex with the color of $x_i$ is $y_{i,0}$, and $x_i y_{i,0}$ is not an edge. Each color class in the $x$-clique has size $q$, and every $y_{i,j}$ has a color different from the vertices of $Y_j$. Hence every internal vertex has at most $q-1$ neighbors of its own color. Moreover, the internal neighbors of an anchor are $x$-vertices and therefore have colors $1$ and $2$.
\end{proof}

\begin{construction}\label{constr:lower}
Fix $D\ge1$, and set
\begin{equation}
   q=D+1,
   \qquad m=18q-12,
   \qquad p=(m+2)q-2,
   \qquad R=2q+1,
   \qquad M=p+6q-2.
   \label{eq:lower-parameters}
\end{equation}
Let $A$ be a set of $p$ vertices. Since $q\ge2$, we may choose $U\subseteq A^*\subseteq A$ with $|U|=mq$ and $|A^*|=(m+1)q-2$. An \emph{anchor block} consists of $A$ together with the following forcers, whose internal vertex sets are pairwise disjoint:
\begin{enumerate}[label=\textup{(F\arabic*)}]
\item a basic $[1,p+1]$-forcer for $A$ under the upper bound $M$;
\item an interval forcer for $U$ under the upper bound $M$, with $R=2q+1$, $a=2$, and $b=2q+1$;
\item an interval forcer for $A^*$ under the upper bound $M$, with $R=2q+1$, $a=q$, and $b=3q-1$.
\end{enumerate}
The required anchor sizes in (F2) and (F3) are, respectively, $p+2-R+1=mq=|U|$ and $p+q-R+1=(m+1)q-2=|A^*|$. Both sizes are at least $2$, both intervals have length $2q=R-1$, and $M\ge p+b$ in each case.

Take three disjoint anchor blocks, with anchor sets $A_1,A_2,A_3$ and corresponding subsets $U_i\subseteq A_i^*\subseteq A_i$. For each $i$, partition $U_i=U_i^1\cup\cdots\cup U_i^m$, where $|U_i^s|=q$ for $1\le s\le m$. Add all edges inside each $U_i^s$, and add every edge between $U_i$ and $U_j$ when $i\ne j$. We denote the resulting graph by $G_D$, put $U^*=U_1\cup U_2\cup U_3$ and $\mathcal A=A_1\cup A_2\cup A_3$, and call $U^*$ the core. Thus $G_D[U_i]$ is the disjoint union of $m$ copies of $K_q$, while $U_1,U_2,U_3$ are pairwise complete.
\end{construction}

\begin{proof}[Proof of Theorem~\ref{thm:lower}]
We use the known forcing and density properties quoted in Lemmas~\ref{lem:basic-forcer}, \ref{lem:interval-forcer}, and~\ref{lem:anchored-density}; these are due to Chen, Mohar and Wu~\cite[Section~4]{ChenMoharWu2023} and Wang, Wang and Yu~\cite[Lemmas~1--4]{WangWangYu2026}. The defective coloring and the core count below are specific to our construction. Fix $D\ge1$ and let $G=G_D$ be the graph in Construction~\ref{constr:lower}. The defective coloring is verified as follows. Give all vertices of $A_i$ color $i-1$, for $i=1,2,3$. The basic forcer is colored with two colors while keeping its anchors in the prescribed color. By Lemma~\ref{lem:forcer-defect-coloring}, we have a coloring of the two interval forcers; after cyclically permuting the colors, their anchors receive the prescribed color. Different attached forcers have disjoint internal vertex sets and no edges between their internal vertices. The new edges between distinct $U_i$ and $U_j$ join different colors. Within $U_i$, the only new same-colored edges form the cliques $U_i^1,\ldots,U_i^m$, each of order $q$. Hence every anchor in $U_i$ has exactly $q-1$ same-colored neighbors in the core and no same-colored neighbor inside an attached forcer. By Lemma~\ref{lem:forcer-defect-coloring}, we obtain the same bound for every other vertex. Thus the coloring has defect $q-1=D$.

It remains to determine the density parameter. Vertices of $\mathcal A\setminus U^*$ are isolated in $G[\mathcal A]$. Every vertex of $U^*$ has $2mq$ neighbors in the other two sets $U_j$ and $q-1$ neighbors in its own $q$-clique. Therefore $\mad(G[\mathcal A])\le\Delta(G[\mathcal A])=(2m+1)q-1\le2p$, since $2p-((2m+1)q-1)=3q-3\ge0$.
By Lemma~\ref{lem:basic-forcer}, every basic forcer has an auxiliary orientation with maximum outdegree at most $p$. For (F2), whose anchor size is $s=mq$, the two quantities in \eqref{eq:forcer-orientation} are $(m+1)q\le p$ and $(m+1)q-1\le p$. For (F3), whose anchor size is $s=(m+1)q-2$, they are $p$ and $p-1$. Lemma~\ref{lem:interval-forcer} therefore gives the required auxiliary orientations of both interval forcers.
Thus every attached forcer has an auxiliary orientation in which its anchors have outdegree zero and every internal vertex has outdegree at most $p$. These orientations agree on $\mathcal A$. Lemma~\ref{lem:anchored-density} gives $\mad(G)\le2p$.

By Construction~\ref{constr:basic-forcer}, each basic forcer contains $K_{p,p(pM+1)}$. Its average degree is $2p\cdot p(pM+1)/(p+p(pM+1))=2p-2p/(pM+2)>2p-2$. Consequently, $\ceil{\mad(G)/2}=p$.

Suppose, for a contradiction, that $G$ has a proper orientation with maximum outdegree at most $M$. Fix $i\in\{1,2,3\}$. By Lemmas~\ref{lem:basic-forcer} and~\ref{lem:interval-forcer}, the three forcers attached to the $i$th anchor block exclude $[1,p+1]$, $[p+2,p+2q+1]$, and $[p+q,p+3q-1]$, respectively. Their union is $[1,p+3q-1]$. Hence every vertex of $U^*$ has outdegree in $\mathcal S=\{0\}\cup[p+3q,p+6q-2]$, and $\mathcal S$ contains exactly $3q$ integers.

The sets $U_1,U_2,U_3$ are pairwise complete, so an outdegree value used on one $U_i$ cannot occur on another $U_j$. Each $U_i$ contains a clique of order $q$ and therefore uses at least $q$ distinct values. Since only $3q$ values are available, each $U_i$ uses exactly $q$ values, the three value sets are disjoint, and their union is $\mathcal S$.

Every component $U_i^s$ is a $q$-clique, so it uses all $q$ values assigned to $U_i$, once each. Since $G[U_i]$ has $m$ such components, every value in $\mathcal S$ occurs exactly $m$ times on $U^*$. Therefore $\sum_{u\in U^*}\dout(u)=m\sum_{x\in\mathcal S}x$. Using $p=(m+2)q-2$, we obtain
\begin{align}
   \sum_{x\in\mathcal S}x
   &=(3q-1)p+\sum_{r=3q}^{6q-2}r \notag\\
   &=\frac{(3q-1)(2mq+13q-6)}2.
   \label{eq:allowed-sum}
\end{align}
On the other hand,
\begin{equation}
   \begin{aligned}
      |E(G[U^*])|
      &=3(mq)^2+3m\binom q2\\
      &=3m^2q^2+\frac{3mq(q-1)}2.
   \end{aligned}
   \label{eq:core-edges}
\end{equation}
Every edge of $G[U^*]$ contributes one to the sum of the outdegrees of vertices in $U^*$, while edges directed from $U^*$ to the attached forcers can only increase that sum. Thus $\sum_{u\in U^*}\dout(u)\ge |E(G[U^*])|$.
However, \eqref{eq:lower-parameters}, \eqref{eq:allowed-sum}, and \eqref{eq:core-edges} give
\begin{align*}
   |E(G[U^*])|-m\sum_{x\in\mathcal S}x
   &=m\bigl(mq-18q^2+14q-3\bigr)\\
   &=m(2q-3)>0.
\end{align*}
This contradicts the outdegree identity above and the preceding inequality. Hence $G$ has no proper orientation with maximum outdegree at most $M$, and $\pchi(G)\ge M+1=p+6q-1$. Since $\ceil{\mad(G)/2}=p$, this yields $\pchi(G)-\ceil{\mad(G)/2}\ge6q-1=6D+5$.
\end{proof}

\subsection{Graphs on surfaces}

Woodall proved defective-coloring and defective-choosability results for graphs on surfaces. His theorem is stated in terms of the Euler characteristic; after rewriting the parameter in terms of Euler genus, the ordinary $3$-coloring part takes the following form.

\begin{theorem}[{Woodall~\cite[Theorem~2(a)(i)]{Woodall2011}}]\label{thm:woodall}
Every graph of Euler genus $\gamma$ has a $3$-coloring with defect at most $D_\gamma=\ceil{\max\{9,2+\sqrt{4\gamma+6}\}}$.
\end{theorem}

We also need the following density estimate, which follows from Euler's formula.

\begin{proposition}[{Surface density estimate; Euler's formula, cf.~\cite{MoharThomassen2001}}]\label{prop:mad-surface}
If $G$ is embeddable on a surface of Euler genus $\gamma\ge1$, then $\mad(G)\le6+\sqrt{6\gamma}$.
\end{proposition}

\begin{proof}
Let $H\subseteq G$ be nonempty and put $n=|V(H)|$. The claim is immediate for $n\le2$. For $n\ge3$, the standard edge bound obtained from Euler's formula~\cite{MoharThomassen2001} gives $|E(H)|\le3n-6+3\gamma$, while simplicity gives $2|E(H)|/n\le n-1$. If $n\le\sqrt{6\gamma}$, then $2|E(H)|/n<\sqrt{6\gamma}$. If $n>\sqrt{6\gamma}$, then $2|E(H)|/n\le6-12/n+6\gamma/n<6+\sqrt{6\gamma}$. Taking the maximum over all nonempty $H\subseteq G$ proves the proposition.
\end{proof}

\begin{proof}[Proof of Theorem~\ref{thm:surface}]
We first assume that $\gamma\ge1$. Woodall's theorem~\cite[Theorem~2(a)(i)]{Woodall2011}, stated as Theorem~\ref{thm:woodall}, gives a partition into three sets, each inducing a graph of maximum degree at most $D_\gamma$. By Proposition~\ref{prop:mad-surface}, $\ceil{\mad(G)/2}\le\ceil{3+\frac12\sqrt{6\gamma}}$. Theorem~\ref{thm:defective-main} now gives $\pchi(G)\le\ceil{3+\frac12\sqrt{6\gamma}}+F(D_\gamma)$, which is $O(\sqrt\gamma)$.

For $\gamma=0$, Chen, Mohar and Wu~\cite[Theorem~1.4]{ChenMoharWu2023} proved the stronger bound $\pchi(G)\le14$. Combining the two cases gives $\pchi(G)\le C\sqrt{\gamma+1}$ for an absolute constant $C$.
\end{proof}

\subsection{Final remarks}

The upper and lower bounds proved above imply
\[
   6\le
   \liminf_{D\to\infty}\frac{f(D)}D
   \le
   \limsup_{D\to\infty}\frac{f(D)}D
   \le\frac{48}{5}.
\]
The lower bound $6D+5$ leaves open both the exact asymptotic coefficient and the lower-order term. The proper $3$-coloring case is a separate finite-parameter problem; the current bounds remain $5\le f(0)\le7$.

The algorithmic assertion in Theorem~\ref{thm:defective-main} assumes that the three-part defective coloring is supplied with the input. Finding such a partition is a different algorithmic problem.

\section*{Acknowledgments}

We acknowledge support from the National Natural Science Foundation of China (Nos.~12371348 and 12201258) and the High-Quality Science and Technology Cultivation Project of Jiangsu Normal University (No.~JSNUGZL2026069).

\section*{Declaration of competing interests}

We declare that we have no competing interests.

\section*{Data availability}

We did not generate or analyze any datasets in this study.

\section*{Declaration on the use of generative AI}

The authors used ChatGPT 5.6 Pro to assist in discussing proof strategies, checking proofs, and improving exposition.

\end{document}